\documentclass[journal]{IEEEtran}
\usepackage{graphicx}
\usepackage[usenames,dvipsnames,svgnames,table]{xcolor}
\definecolor{steelblue}{RGB}{70,130,180}
\definecolor{warmred}{RGB}{244, 104, 65}
\usepackage[cmex10]{amsmath}
\usepackage{amssymb,amsfonts,mathrsfs,mathtools,amsthm}
\usepackage{epstopdf}
\usepackage{lscape}
\usepackage{bm}
\usepackage{cite}
\usepackage{url}
\usepackage{subfigure}
\usepackage[colorlinks=true,allcolors=blue]{hyperref}
\usepackage{graphicx}
\usepackage{color}

\usepackage{algorithm,algorithmicx,algpseudocode}

\definecolor{steelblue}{RGB}{70,130,180}
\definecolor{warmred}{RGB}{244, 104, 65}
\IEEEoverridecommandlockouts

\newtheorem{assumption}{Assumption}
\newtheorem{remark}{Remark}
\newtheorem{definition}{Definition}
\newtheorem{theorem}{Theorem}

\begin{document}
\title{Optimal Power Flow with State Estimation \\In the Loop for Distribution Networks}

\author{Yi Guo,~\IEEEmembership{Member,~IEEE,}
Xinyang Zhou,~\IEEEmembership{Member,~IEEE,}
Changhong Zhao,~\IEEEmembership{Senior Member,~IEEE,}\\
Lijun Chen,~\IEEEmembership{Member,~IEEE,} 
and Tyler Holt Summers,~\IEEEmembership{Member,~IEEE}
\thanks{This material is based on work supported by funding from US Department of Energy Office of Energy Efficiency and Renewable Energy Solar Energy Technologies Office under contract No. DE-EE-0007998 and the National Science Foundation under grant CMMI-1728605. This work was partially supported by an ETH Z\"{u}rich Postdoctoral Fellowship and Hong Kong RGC Early Career Award No. 24210220.}
\thanks{Y. Guo is with the Power Systems Laboratory, ETH Z\"urich, Z\"urich, 8092, Switzerland, email: yiguoy@ethz.ch.}
\thanks{X. Zhou is with Department of Power System Engineering, National Renewable Energy Laboratory, Golden, CO 80401, USA, email: xinyang.zhou@nrel.gov.}
\thanks{C. Zhao is with Department of Information Engineering, the Chinese University of Hong Kong, HKSAR, China, email: chzhao@ie.cuhk.edu.hk.}
\thanks{L. Chen is with College of Engineering and Applied Science, The University of Colorado, Boulder, CO 80309, USA, email: lijun.chen@colorado.edu.}
\thanks{T. Summers is with the Department
of Mechanical Engineering, The University of Texas at Dallas, Richardson, TX 75080, USA, email: tyler.summers@utdallas.edu.}
}
\maketitle

\begin{abstract}
We propose a framework for running optimal control-estimation synthesis in distribution networks. Our approach combines a primal-dual gradient-based optimal power flow (OPF) solver with a state estimation (SE) feedback loop based on a limited set of sensors for system monitoring, instead of assuming exact knowledge of all states. The estimation algorithm reduces uncertainty on unmeasured grid states based on certain online state measurements and noisy ``pseudo-measurements''. We analyze the convergence of the proposed algorithm and quantify the statistical estimation errors based on a weighted least squares (WLS) estimator. The numerical results on a 4,521-node network demonstrate that this approach can scale to extremely large networks and provide robustness to both large pseudo-measurement variability and inherent sensor measurement noise.
\end{abstract}

\begin{IEEEkeywords}
optimal power flow, state estimation, feedback control, large-scale networks, voltage regulation, distribution networks and power systems.
\end{IEEEkeywords}

\section{Introduction}
\IEEEPARstart{T}{he} increasing penetration of distributed energy resources (DERs) has provided an opportunity to explore the benefits of advanced smart grid technologies in  distribution networks. As the heterogeneous control strategies of grid-connected elements dominate distribution networks, many customers will become active and motivated end-users to optimize their own power usage via optimal power flow (OPF) methods \cite{opf,dommel,alsac,Baldick,low1}. This requires the power system control scheme to have real-time knowledge about the structure and state of the distribution network (e.g., operation states, net-loads variation, device dynamics, network topology, etc.), and to provide the corresponding real-time responses (e.g., optimal control inputs. set-points of DERs, etc.) for safe and efficient operation. However, the current distribution network control paradigm cannot satisfy the above requirement due to lack of real-time knowledge about the system state, and high expense of real-time state measurement. Future distribution systems will require more sophisticated and tightly integrated control, optimization, and estimation methods for these issues.

Most OPF methods for distribution networks in the literature assume complete availability of network states to implement various optimal control strategies \cite{summers2015stochastic,molzahn2019survey,wu1990power,abido2002optimal,wang2020asynchronous}. However, in practice network states must be estimated with a monitoring system from noisy measurements, which itself is a challenging problem due to the increasingly complex, extremely large-scale, and nonlinear time-varying nature of emerging networks. To solve these issues, the recently proposed OPF frameworks \cite{colombino2019online,zhou2017incentive,dall2016optimal,li2017distribued,bernstein2019real,bolognani2014distributed} leverage measurement feedback-based online optimization method to loop the physical measurement information back to OPF controllers, which adapt the OPF decisions to real-time data to mitigate the effects of inherent disturbances and modelling errors. It is unrealistic to have real-time physical measurements of system states at every corner of distribution networks due to heavy communication burdens, end-user privacy concerns, and high costs.

In this paper, we propose a more general framework than the existing OPF approaches, which tightly integrates state estimation (SE) techniques \cite{liu2012trade,dehghanpour2019game,abur2004power,dehghanpour2018survey,schweppe1970power} into online OPF control algorithms for distribution networks. This OPF with SE in the loop framework allows us to utilize a limited set of sensor measurements together with a power system state estimator instead of exact knowledge of network states. The power system state estimator, which may include data from the Supervisory Control and Data Acquisition (SCADA) system, phasor measurement units (PMUs), topology processor and pseudo-measurements\footnote{Due to the lack of real time measurement and stochasticity nature of power net-loads in distribution system state estimation, the nodal power injections are measured by their nominal load-pattern (i.e., the real value plus zero-mean random deviations), so-called pseudo-measurement, whose information is derived from the past records of load behaviors \cite{schweppe1970power}.}, provides the best available information about network states \cite{primadianto2016review, ahmad2018distribution,dehghanpour2019game,dehghanpour2018survey,liu2012trade,zamzam2019data} and in-turn enables implementation and enhances the performance of OPF controllers. Our approach allows OPF decisions to adapt to real-time time-varying stochastic DERs and loads, and compensates for disturbances and modelling errors, since SE results utilize measurement data from the actual nonlinear system response. 

A preliminary version of this work appeared in \cite{guo2019solving}, and here we significantly expand the work in several directions into the present paper. 
Our main contributions are summarized as follows:
\begin{itemize}
\item[1)] We formulate a general convex OPF problem subject to power flow equations and network-wise coupling constraints. To integrate OPF with SE in the loop, we propose a primal-dual gradient-based OPF algorithm with state estimation feedback. Instead of requiring full knowledge of all system states, the controller utilizes at every gradient step real-time monitoring information from state estimation results to inform control decisions. Although OPF and SE problems for distribution networks have been widely studied individually, none of the existing literature reveals the connection and bridges the gap between them. Here we are closing the loop between OPF and state estimation in large-scale distribution networks by utilizing state estimates \cite{gomez2004power,wu1990power}. This allows us to react to real-time information of system states with a limited number of deployed sensors. In principle, the proposed framework is compatible with a variety of state estimation methods and control strategies in distribution networks. Here, we illustrate the approach through a voltage regulation problem, with voltage magnitude estimation in the loop. 

\item[2)] We leverage linear approximations to the AC power flow equations to facilitate scalable and computationally efficient OPF formulation for SE feedback integration \cite{dall2016optimal,colombino2019online,zhou2017incentive}. The voltage profile estimation uses a weighted least squares (WLS) estimator. Convergence of the proposed gradient-based algorithm with state estimation feedback is analyzed. Additionally, we quantify the statistical estimation errors of the WLS estimator and show how the errors influence the theoretical results and convergence proofs in \cite{guo2019solving}. This analysis provides a measure of quality of the SE feedback associated with a particular allocation of sensors across the network. 

\item[3)]The effectiveness, scalability, flexibility and robustness of the proposed algorithm are demonstrated on a 4,521-node multi-phase unbalanced distribution network with 1043 (aggregated) net-loads. With only 3.6\% voltage measurement deployment, the integrated OPF controller with SE feedback effectively regulates network voltage. The distributed algorithm using linearized distribution flow (LinDistFlow) enables scaling to extremely large networks. The numerical results also indicate that the proposed OPF controller with SE feedback is robust to the inherent measurement noise and estimation errors. 


\end{itemize}

The rest of this paper is organized as follows. Section~\ref{sec:idea} discusses the general concept of OPF with SE in the loop for distribution networks. Section~\ref{sec:model} formulates an OPF problem and introduces gradient algorithm with state estimation feedback. Section~\ref{sec:num} demonstrates numerical results on voltage regulations and Section~\ref{sec:con} concludes.

\section{Optimal Power Flow with State Estimation in the loop}\label{sec:idea}
In this section, we propose an OPF solver with state estimation feedback. We first pose a general problem to highlight the overall approach, and in subsequent sections we detail the model, objectives, constraints and state estimator for certain control and monitoring purposes. 

Consider an OPF problem for distribution networks:
\begin{subequations}\label{eq:opt}
\begin{eqnarray}
& \underset{\mathbf{p},\mathbf{q}}{\min} & \sum_{i\in\mathcal{N}}C_i(p_i,q_i)+ C_0(\mathbf{p},\mathbf{q}),\\
&\text{s.t.} &\mathbf{g}(\mathbf{r}(\mathbf{p},\mathbf{q})) \leq 0, \label{eq:voltreg}\\
&& (p_i,q_i)\in\mathcal{Z}_i,\forall i\in\mathcal{N},\label{eq:X} 
\end{eqnarray}
\end{subequations}
where $C_0(\mathbf{p},\mathbf{q})$ is a cost function capturing system objectives (e.g., cost of deviation of total power injections into the substation from preferred values) and the local objective functions $C_i(p_i,q_i)$ capture (a weighted mix of) generation costs, ramping costs, active power losses, renewable curtailment penalty, auxiliary service expenses, reactive compensation (comprising a weighted sum thereof), etc. at node $i \in \mathcal{N}$. 

We then define a vector $\mathbf{r(\mathbf{p},\mathbf{q})} \in \mathbb{R}^{N_r}$ collecting the electrical quantities of an operator's interest (e.g., voltage magnitudes, current injections, power injection at the substation, etc.), which depends on nodal power injections $\mathbf{p} :=[p_1,\dots,p_N]^\top$ and $\mathbf{q} :=[q_1,\dots,q_N]^\top$ through power flow equations.\footnote{The power flow equations may model nonlinear AC power flow as well as its SOCP relaxations, SDP relaxations, or various linearizations. The general approach of OPF with SE in the loop can be adapted to various power flow models. In the rest of this paper, we use a linearized power flow model to illustrate the effectiveness of the proposed framework.} The function $\mathbf{g}:\mathbb{R}^{N_r}\rightarrow \mathbb{R}^{N_g}$ models network constraints, including those on voltage magnitudes and angles, current injections, and line flows. The nodal power injections $(\mathbf{p},\mathbf{q})$ are constrained within convex and compact feasible sets $\mathcal{Z}_i$.

\begin{figure}[!tbhp]
    \centering
    \includegraphics[width=3.5in]{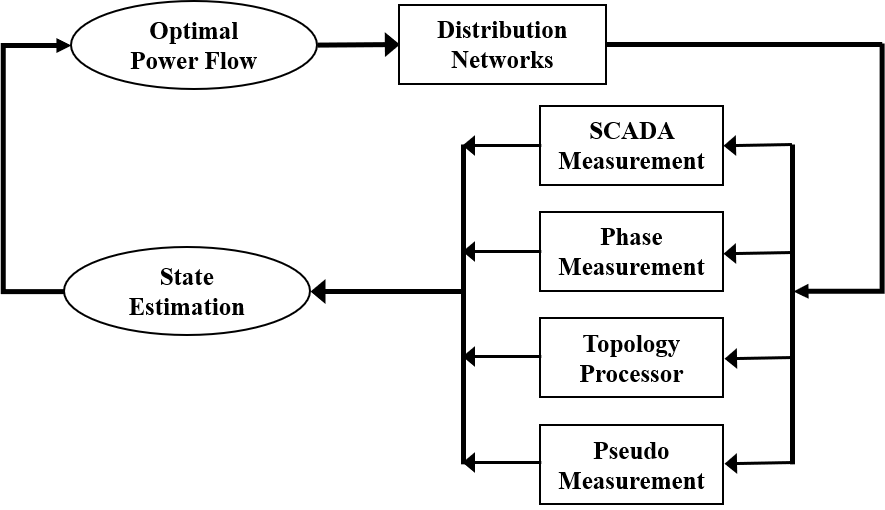}
    \caption{The concept of solving optimal power flow with state estimation in the loop.}
    \label{fig:OPF_SE_Loop}
\end{figure}

Problem~\eqref{eq:opt} is typically solved assuming that all needed system states are available. However, in practice, there is generally a lack of reliable measurement devices and communication infrastructure in distribution networks, rendering most system states directly unavailable and thus the conventional OPF approaches unimplementable. Therefore, the main challenges for solving \eqref{eq:opt} in practice lie in how to best integrate the \emph{estimates} of current system states $\mathbf{r}$ and understand tradeoffs between sensor deployment cost, SE performance, and OPF controller performance. We will tackle these challenges by integrating a state estimation feedback loop with a limited number of sensor measurements into OPF solvers and analytically characterize the overall performance thereof. This allows the OPF controller to respond to real-time information and update control decisions despite nodes without measurement in the grid. The overall approach is illustrated in Fig.~\ref{fig:OPF_SE_Loop}.

In our general framework, the SE techniques may determine the system states using any or all of SCADA measurements, phasor measurement units (PMUs) measurements, pseudo-measurements and topology information to reduce estimation uncertainty. How to fuse different sources of information into OPF formulations remains largely an open question under exploration. We aim to indicate that there are many possibilities and research direction to potentially improve control and optimization in distribution networks through tight integration with SE. The increasing penetration of renewable energy resources and distributed generators enable the distribution networks with smart features, such as demand response and distributed automation. This allows the networks turn to a more active and complex system with fast system response. An efficient, real-time monitoring of distribution networks should be looped into OPF controllers. In the rest of this paper, we take the voltage regulation problem with voltage estimation as an illustrative example, which is based on a few PMU voltage measurements and net-loads pseudo-measurements. 


\section{Gradient-Based OPF Solver with State Estimation Feedback}\label{sec:model}
\subsection{System Modelling}
Consider a distribution network denoted by a directed and connected graph $\mathcal{G}(\mathcal{N}_0,\mathcal{E})$, where $\mathcal{N}_0:= \mathcal{N}\cup\{0\}$ is a set of all ``buses" or ``nodes" with substation node 0 and $\mathcal{N}:= \{1,\dots,N\}$, and $\mathcal{E} \subset \mathcal{N}\times\mathcal{N}$ is a set of ``links" or ``lines" for all $(i,j) \in \mathcal{E}$. Let $V_i :=|V_i|e^{j\angle V_i} \in \mathbb{C}$ and $I_i:= |I_i|e^{j\angle I_i}\in \mathbb{C}$ denote the phasor for the line-to-ground voltage and the current injection at node $i\in\mathcal{N}$. The absolute values $|V_i|$ and $|I_i|$ denote the signal root-mean-square values and $\angle V_i$ and $\angle I_i$ corresponding to the phase angles with respect to the global reference. We collect these variables into complex vectors $\mathbf{v}:=[V_1,V_2,\dots,V_N]^\top \in\mathbb{C}^N$ and $\mathbf{i}:=[I_1,I_2,\dots,I_N]^\top \in \mathbb{C}^N$. We denote the complex admittance of line $(i,j)\in\mathcal{E}$ by $y_{ij} \in \mathbb{C}$. The admittance matrix $\mathbf{Y}\in \mathbb{C}^{N\times N}$ is given by
\begin{equation}\label{admittanceMatrix}
\mathbf{Y}_{ij} = \left\{ \begin{array}{ll}
\sum_{l \sim i} y_{il} + y_{ii} & \textrm{if $i=j$}\\
-y_{ij} & (i,j) \in \mathcal{E} \\
0 &  (i,j) \notin \mathcal{E}
\end{array} \right .,
\end{equation}
where $l \sim i$ indicates connection between node $l$ and node $i$, and $y_{ii}$ is the self admittance of node $i$ to the ground.

Node 0 is modelled as a slack bus. The other nodes are modelled as PQ buses for which the injected complex power is specified. The admittance matrix can be partitioned as
\begin{equation}\nonumber
\begin{bmatrix}
I_0\\
\mathbf{i}
\end{bmatrix} =
\begin{bmatrix}
y_{00} & \bar{y}^\top  \\
\bar{y} & \mathbf{Y}
\end{bmatrix}
\begin{bmatrix}
V_0\\
\mathbf{v}
\end{bmatrix}.
\end{equation}
The net complex power injection then reads:
\begin{equation}\label{powerbalancewithPCC}
\mathbf{s} = \textrm{diag}(\mathbf{v})\Big(\mathbf{Y}^*(\mathbf{v})^* + \bar{y}^*(V_0)^* \Big),
\end{equation}
where superscript $(\cdot)^*$ indicates the element-wise conjugate of complex vector $\mathbf{\mathbf{v}}$.

To facilitate computational efficiency using convex optimization, here we leverage a linearization of \eqref{powerbalancewithPCC} as follows: 
\begin{equation}
\label{linear_powerflow}
    \mathbf{r} = \mathbf{A}\mathbf{p} + \mathbf{B}\mathbf{q} + \mathbf{r}_0,
\end{equation}
where the parameters $\mathbf{A}$, $\mathbf{B}$ and $\mathbf{r}_0$ can be attained from various linearization methods, e.g., \cite{bernstein2017linear,gan2016online}. In the rest of this paper, the linearized coefficient matrices $\mathbf{A}$ and $\mathbf{B}$ are fixed over times for simplification. From now on, we limit $\mathbf{r}(\mathbf{p},\mathbf{q})$ to the above linearized power flow model \eqref{linear_powerflow}. Recall that $\mathbf{r} \in \mathbb{R}^{N_r}$ represents certain electrical quantities of an operator's interest (e.g., voltage magnitudes, current injections, power injection at the substation and so on). 


\subsection{OPF Formulation and Primal-Dual Gradient Algorithm}
In this section, we introduce a general OPF problem and the pertinent gradient algorithm with idealized measurement feedback\footnote{The ``idealized" refers to the full measurement of vector $\mathbf{r}$ without noise.} from nonlinear power flow to reduce modelling errors. The feasible operating regions $\mathcal{Z}_i$ depend on the terminal properties of various dispatchable devices, e.g., inverter-based distributed generators, energy storage systems or small-scale diesel generators. We make the following assumptions.

\begin{assumption}[Slater's condition]\label{assumption_slater}
There exists a strictly feasible point within the operation region $(\mathbf{p},\mathbf{q}) \in \mathcal{Z}$, where $\mathcal{Z}:=\mathcal{Z}_1 \times\ldots\times \mathcal{Z}_N$, so that
\begin{equation}
    \mathbf{g}(\mathbf{r}(\mathbf{p},\mathbf{q})) < 0.\nonumber
\end{equation}
\end{assumption}

\noindent Assumption~\ref{assumption_slater} suffices strong duality for problem \eqref{eq:opt}.

\begin{assumption}\label{assumption_3}
A set of local objective functions $C_i(p_i,q_i), \forall i\in \mathcal{N}$ are continuously differentiable\footnote{The continuous differentiability of objective functions is not necessary to have the bound in Theorem 2, but it contributes to the convergence performance in a large-scale network system.} and strongly convex as functions of $(p_i,q_i)$, and their first order derivatives are bounded within their operation regions indicated as $(p_i,q_i) \in \mathcal{Z}_i, \forall i\in \mathcal{N}$; The system-wise objective function $C_0(\mathbf{p},\mathbf{q})$ is continuously differentiable and convex with its first-order derivative bounded. Furthermore, the constraint function $\mathbf{g}$ is continuously differentiable and convex with bounded derivatives on its domain.
\end{assumption}

Assumption~\ref{assumption_3} guarantees decent and convenient properties in performance analysis and is usually satisfied in real-world implementations, e.g., quadratic cost functions and linear constraints.

We write the regularized Lagrangian
$\mathcal{L}$ for \eqref{eq:opt} as follows:
\begin{equation}\label{eq:L_opf}
\mathcal{L} = \sum_{i\in\mathcal{N}}C_i(p_i,q_i) + C_0(\mathbf{p},\mathbf{q}) + \bm{\mu}^\top \mathbf{g}(\mathbf{r}(\mathbf{p},\mathbf{q})) - \frac{\eta}{2}\|\bm{\mu}\|_2^2,
\end{equation}
where $\bm{\mu}\in\mathbb{R}_{+}^{N_\mu}$ is the dual variable vector associated with the general inequality constraints and we keep the feasible regions $\bm{\mu}\in\mathbb{R}_{+}^{N_\mu}$ and $(\mathbf{p},\mathbf{q})\in\mathcal{Z}$ implicit. To facilitate proof of convergence, the Lagrangian \eqref{eq:L_opf} includes a Tikhonov regularization term $-\frac{\eta}{2}\|\bm{\mu}\|_2^2$ with a prescribed small parameter $\eta$ that introduces bounded discrepancy. The upshot of having a regularized term in Lagrangian \eqref{eq:L_opf} is that the gradient-based approaches can be applied to \eqref{eq:L_opf} to find an approximate solution of the original Lagrangian with improved convergence properties. The discrepancy due to the regularized term is discussed and quantified in \cite{koshal2011multiuser}.

To solve \eqref{eq:opt}, we come to the following saddle-point problem:
\begin{equation}\label{eq:maxmin_L}
    \max_{\bm{\mu}\in\mathbb{R}^{N_\mu}_{+}} \min_{(\mathbf{p},\mathbf{q})\in\mathcal{Z}} \mathcal{L} \left(\mathbf{p}, \mathbf{q},\bm{\mu} \right),
\end{equation}
which can be solved by an iterative primal-dual gradient algorithm to reach the unique saddle-point of \eqref{eq:maxmin_L} as:
\begin{subequations}\label{eq:primaldual}
\begin{eqnarray}
& \mathbf{r}^{k}&=\mathbf{A}\mathbf{p}^{k} + \mathbf{B}\mathbf{q}^{k} + \mathbf{r}_0,\label{eq:gradient_opf_pf}\\
& \begin{bmatrix}
\mathbf{p}^{k+1}\\
\mathbf{q}^{k+1}
\end{bmatrix} & =\begin{bmatrix}
\mathbf{p}^k -\epsilon \nabla_{\mathbf{p}}\mathcal{L}(\mathbf{p}^k,\mathbf{q}^k,\bm{\mu}^k)\\
\mathbf{q}^k -\epsilon \nabla_{\mathbf{q}}\mathcal{L}(\mathbf{p}^k,\mathbf{q}^k,\bm{\mu}^k)
\end{bmatrix}_{\mathcal{Z}},\label{eq:gradient_opf_pq}\\
& \bm{\mu}^{k+1}&=\left[\bm{\mu}^k + \epsilon \nabla_{\bm{\mu}}\mathcal{L}(\mathbf{r}^k,\bm{\mu}^k)\right]_{\mathbb{R}_+^{N_\mu}},\label{eq:gradient_opf_dual}
\end{eqnarray}
\end{subequations}
where $\epsilon >0$ is a constant stepsize to be determined and the operator $[\cdot]_{\mathcal{Z}}$ project onto the feasible set $\mathcal{Z}= \times_{i\in\mathcal{N}}\mathcal{Z}_{i}$ \footnote{We use the projection operator $[\cdot]_{\mathcal{Z}}$ in primal updates instead of $[\cdot]_{\mathcal{Z}_p}$ and $[\cdot]_{\mathcal{Z}_q}$, since the feasible sets of active and reactive power are not independent, but correlated based on the terminal apparent power limits.}. The operator $[\cdot]_{\mathbb{R}_+^{N_\mu}}$ projects onto nonnegative orthant.

We here introduce a compact mapping
\begin{equation}\label{eq:linear_operator}\nonumber
    \Phi:\{\mathbf{p}^k,\mathbf{q}^k, \bm{\mu}^k\} \mapsto  \begin{bmatrix}  \nabla_{\mathbf{p}}\mathcal{L}\left(\mathbf{p}^k,\mathbf{q}^k,\bm{\mu}^k\right) \\
     \nabla_{\mathbf{q}}\mathcal{L}\left(\mathbf{p}^k,\mathbf{q}^k,\bm{\mu}^k\right)\\
    -  \nabla_{\mathbf{\bm{\mu}}}\mathcal{L}\left(\mathbf{r}^k(\mathbf{p}^k,\mathbf{q}^k),\bm{\mu}^k\right)
    \end{bmatrix},
\end{equation}
so that \eqref{eq:primaldual} can be rewritten as:
\begin{equation}
    \mathbf{x}^{k+1} = \left[\mathbf{x}^k -  \epsilon \Phi(\mathbf{x}^k)\right]_{\mathbb{R}_+^{N_\mu} \times \mathcal{Z}},
\end{equation}
where $\mathbf{x}^k :=[(\mathbf{p}^k)^\top,(\mathbf{q}^k)^\top,(\bm{\mu}^k)^\top]^\top$. 
Under Assumption \ref{assumption_3}, it can be shown \cite{zhou2019accelerated} that $\Phi$ is strongly monotone and Lipschitz continuous, i.e., it satisfies for all feasible points $\mathbf{x}_1$ and $\mathbf{x}_2$ and for some constants $M>0$ and $L>0$ we have:
\begin{equation}\label{eq:property_Monotone}
\left(\Phi(\mathbf{x}_1) - \Phi(\mathbf{x}_2)\right)^\top \left(\mathbf{x}_1 - \mathbf{x}_2\right) \geq M\|\mathbf{x}_1 - \mathbf{x}_2\|_2^2,
\end{equation}
\begin{equation}\label{eq:property_Lipschitz}
            \|\Phi(\mathbf{x}_1) - \Phi(\mathbf{x}_2) \|_2^2 \leq L^2 \|\mathbf{x}_1 - \mathbf{x}_2\|_2^2.
\end{equation}
We now have the following convergence results.

\begin{theorem} \label{theorem_1}
Under Assumptions \ref{assumption_slater} and \ref{assumption_3},
if the step size $\epsilon$ satisfies
\begin{equation} \label{stepsizecond}
0 < \epsilon < 2M/L^2,
\end{equation}
then algorithm \eqref{eq:primaldual} converges to the unique saddle point of \eqref{eq:maxmin_L}.
\end{theorem}
The proof of Theorem \ref{theorem_1} is skipped because it follows the method in \cite{zhou2019accelerated,guo2019solving} based on strong monotonicity and Lipschitz continuity of operator $\Phi(\cdot)$.

\subsection{Feedback-Based Implementation}
The optimization problem \eqref{eq:opt} and gradient algorithm \eqref{eq:primaldual} are based on a linearized power flow to guarantee its convexity and prove convergence to the saddle point. However, linearization errors cause the solution of \eqref{eq:primaldual} to be suboptimal or even infeasible for the system with nonlinear power flow. To address this issue, feedback-based online optimization methods \cite{colombino2019online,bernstein2015composable,zhou2017incentive} have been leveraged to reduce the effects of modelling error. In particular, by replacing \eqref{eq:gradient_opf_pf} with the following nonlinear power flow model
\begin{eqnarray}\label{eq:nonlinear_feedback}
\tilde{\mathbf{r}}^{k} &=& f(\mathbf{p}^k,\mathbf{q}^k),
\end{eqnarray}
a more accurate measurement $\tilde{\mathbf{r}}^{k}$ can be used instead of an approximate model, to update the dual variables in \eqref{eq:gradient_opf_dual}. This facilitates a real-time implementation to track the time-varying grid response, with proofs of convergence to a bounded region surrounding the optimum \cite{dall2016optimal,zhou2017incentive}.

To this end, the following information availability and communication structure are usually applied to conduct \eqref{eq:primaldual} in practice. A network coordinator is responsible for maintaining operational constraints by actively monitoring the values of $\mathbf{r}$ and updating $\bm{\mu}$ according to \eqref{eq:gradient_opf_dual}. The network coordinator usually does not know the exact values of $\mathbf{p}$ and $\mathbf{q}$ at local DERs, so the updating of \eqref{eq:gradient_opf_pq} is executed by DERs in a distributed way.

However, another crucial issue of such feedback-based algorithm has been largely overlooked: in practice, there are too few monitoring devices in distribution systems to measure all components of $\mathbf{r}$, and therefore it is not possible to directly implement feedback-based algorithms to solve the problem \eqref{eq:opt}. Our preliminary results \cite{guo2019solving} demonstrated that limited knowledge of system states can lead the OPF controller to cause constraint violations. 

To enable an implementation of feedback-based OPF algorithms in distribution networks, and also to improve performance of algorithms that make use of ``pseudo-measurements'', we integrate a state estimation algorithm based on a sparse set of available measurements, before performing the dual variable update \eqref{eq:gradient_opf_dual}. This allows us to utilize improved information on the network state to make decisions, specifically improving information at nodes without measurement of the grid where there are no direct measurements. Fig.~\ref{fig:OPF_SE_diagram_1} illustrates the proposed OPF framework with state estimation in the loop. 

\begin{figure}[!htbp]
\centering
\includegraphics[width=3.5in]{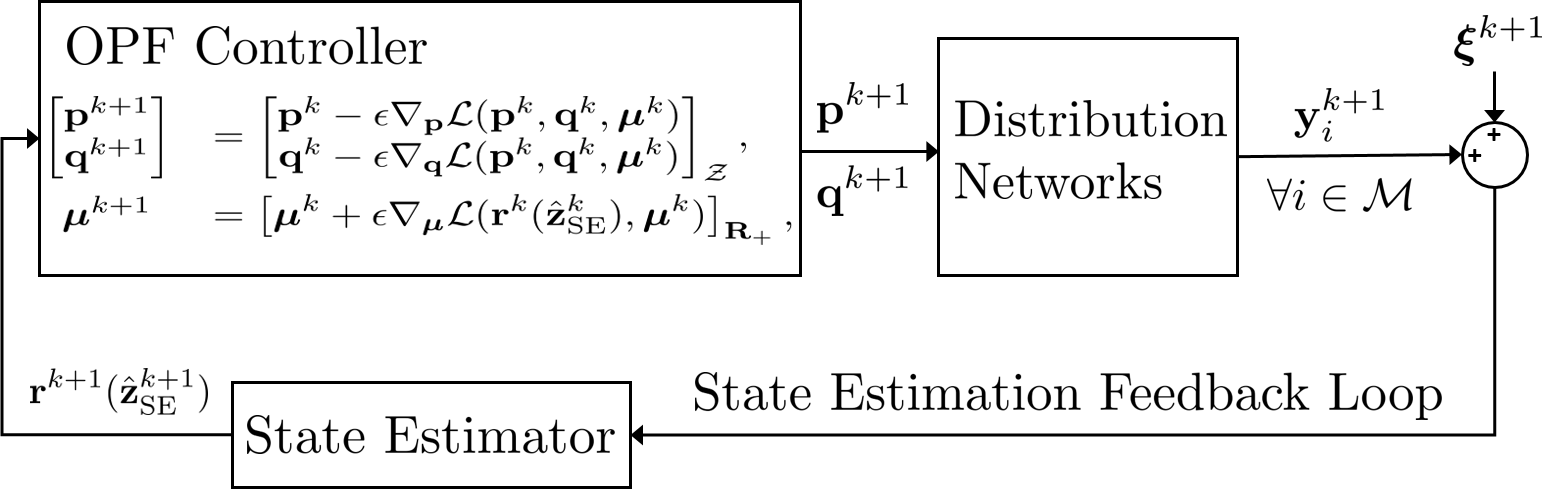}
\caption{The diagram of the proposed optimal power flow problem with SE in the loop.}
\label{fig:OPF_SE_diagram_1}
\end{figure}

\subsection{State Estimation In the Loop}
We define $\mathbf{z}^{k}=[(\mathbf{p}^{k})^\top,(\mathbf{q}^{k})^\top]^{\top}$ as the system states at iteration $k$, and the grid measurement model
\begin{equation} \label{eq:measurements}
    \mathbf{y}^k = \mathbf{h}(\mathbf{z}^{k}) + \bm{\xi}^k,
\end{equation}
where $\mathbf{y}^k\in\mathbb{R}^{N_y}$ is a measurement vector received at iteration $k$ comprising raw noisy measurements from sensors (e.g., voltage magnitude and angles) and pseudo-measurements (e.g., real and reactive power injections at load buses), and measurement function $\mathbf{h}: \mathbb{R}^{N_z} \to \mathbb{R}^{N_y}$. The vector $\bm{\xi}^k\in\mathbb{R}^{N_y}$ models measurement errors, which are assumed to be independently and identically distributed Gaussian noise $\mathbb{N}(0, \bm{\Sigma})$ where $\bm{\Sigma} \in \mathbb{R}_+^{N_y \times N_y}$ is a diagonal matrix. The pseudo-measurements are modeled as sensor measurements corrupted by high-variance Gaussian noise based on historical data (e.g., customer billing data and typical load profile) that provide rough information about variations in the state of the grid \cite{gomez2004power}.

To estimate grid states from the available measurements, we consider the WLS estimator \cite{gomez2004power,wu1990power,zhou2020gradient}:
\begin{equation}\label{eq:se_general}
\hat{\mathbf{z}}^k_{\textrm{SE}} = \underset{\mathbf{z}^k}{\text{argmin}}  \frac{1}{2}\left(\mathbf{y}^k -\mathbf{h}(\mathbf{z}^k)\right)^\top \mathbf{W}\left(\mathbf{y}^k - \mathbf{h}(\mathbf{z}^k)\right),
\end{equation}
where the weight matrix is defined as $\mathbf{W} = \bm{\Sigma}^{-1}$. The estimated quantities of interest $\mathbf{r}^k(\hat{\mathbf{z}}_{\textrm{SE}}^k)$ is uniquely determined by the states $\hat{\mathbf{z}}^k_{\textrm{SE}}=[(\hat{\mathbf{p}}^{k}_{\textrm{SE}})^\top,(\hat{\mathbf{q}}^{k}_\textrm{SE})^\top]^{\top}$ through power flow equations. Existence and uniqueness of a solution to \eqref{eq:se_general} require certain properties of the measurement function.
\begin{definition}[Full Observability\footnote{This definition should be distinguished from observability of linear
dynamical systems. Here, we limit the definition of observability to power system static
state estimation problems \cite{schweppe1970power} throughout this manuscript.}\cite{abur2004power,wu1985network}] A state-to-output system $\mathbf{y}=\mathbf{h}(\mathbf{z})$ is fully observable if $\mathbf{z} = 0$ is the only solution for $h(\mathbf{z}) = 0$, which allows a unique solution to \eqref{eq:se_general}.
\end{definition}
\begin{assumption}\label{assumption:full_observability} 
The system \eqref{eq:measurements} is fully observable by using pseudo-measurement over all nodes.
\end{assumption}


Since distribution networks typically have only a sparse set of real-time measurements from deployed sensors, we use pseudo-measurements over all nodes to ensure full observability. The combinations of real measurements and pseudo-measurements have been observed to be effective in \cite{zhou2020gradient, dvzafic2013real}.
Overall, distribution network state estimation methods have been widely discussed and shown to be accurate and computationally efficient under nominal operating conditions of distribution networks \cite{primadianto2016review}.




Fig.~\ref{fig:OPF_SE_diagram_1} and Algorithm \ref{algorithm:OPF-SE algorithm} illustrate and describe the proposed OPF controller with the state-estimation feedback loop. 
\begin{algorithm}
    \caption{(OPF with SE in the loop)}\label{algorithm:OPF-SE algorithm}
    \begin{algorithmic}[1]
        \Require net-loads initialization $(\mathbf{p}^0,\mathbf{q}^0)$ and $\bm{\mu}^k$
        \For{$k = 0:K$}
            \State $\tilde{\mathbf{r}}^{k} \leftarrow  \textrm{nonlinear power flow}  ~(\mathbf{p}^{k}, \mathbf{q}^{k})$
            \State receive system measurement $\mathbf{y}^{k}$ 
            \State estimate the system states $\hat{\mathbf{z}}^k_{\textrm{SE}}$ based on $\mathbf{y}^{k}$ and calculate the electrical quantities of interest $\mathbf{r}^k(\hat{\mathbf{z}}^k_{\textrm{SE}})$ 
            \State $\begin{bmatrix}
            \mathbf{p}^{k+1}\\
            \mathbf{q}^{k+1}
            \end{bmatrix} = \begin{bmatrix}
            \mathbf{p}^k - \epsilon \nabla_{\mathbf{p}} \mathcal{L}(\mathbf{p}^k,\mathbf{q}^k,
            \bm{\mu}^k, \mathbf{r}^k(\hat{\mathbf{z}}^k_{\textrm{SE}}) )\\
            \mathbf{q}^k - \epsilon \nabla_{\mathbf{q}} \mathcal{L}(\mathbf{p}^k,\mathbf{q}^k,
            \bm{\mu}^k,\mathbf{r}^k(\hat{\mathbf{z}}^k_{\textrm{SE}}) )
            \end{bmatrix}_{\mathcal{Z}}$
            \State $\bm{\mu}^{k+1} =  \left[\bm{\mu}^{k} + \epsilon\nabla_{\bm{\mu}} \mathcal{L}(\mathbf{r}^k(\mathbf{z}^k_{\textrm{SE}}),\bm{\mu}^k)\right]_{\mathbb{R}_+^{N_\mu}}$
        \EndFor
\end{algorithmic}
\end{algorithm}
Note that the step 2 in Algorithm \ref{algorithm:OPF-SE algorithm} is not implemented in the proposed OPF controller, but instead is a physical response of the power system to the up-to-date nodal power injections $(\mathbf{p}^k, \mathbf{q}^k)$. We utilize SE in the loop to compute a state estimate $\hat{\mathbf{r}}^k$, which then contributes to the update of dual variables $\mathbf{\bm{\mu}}^{k+1}$ in step 6. Our numerical experiments in Section \ref{sec:num} compare this approach with the direct use of noisy measurements and pseudo-measurements without any estimation scheme.
\subsection{Convergence Analysis}

To analyze the impact of the power flow model adopted by the proposed framework, we first define $\tilde{\Phi}(\cdot)$ as the gradient mapping based on a nonlinear power flow model:
\begin{equation}\nonumber
    \tilde{\Phi}:\{\mathbf{p}^k,\mathbf{q}^k, \bm{\mu}^k\} \mapsto  \begin{bmatrix}
     \nabla_{\mathbf{p}}\mathcal{L}\left(\mathbf{p}^k,\mathbf{q}^k,\bm{\mu}^k\right) \\
     \nabla_{\mathbf{q}}\mathcal{L}\left(\mathbf{p}^k,\mathbf{q}^k,\bm{\mu}^k\right)\\
    -  \nabla_{\mathbf{\bm{\mu}}}\mathcal{L}\left(\tilde{\mathbf{r}}^k(\mathbf{p}^k,\mathbf{q}^k),\bm{\mu}^k\right)
    \end{bmatrix},
\end{equation}
where different from the mapping $\Phi$ defined for \eqref{eq:primaldual}, the linearized power flow $\mathbf{r}^k(\mathbf{p}^k, \mathbf{q}^k)$ in \eqref{eq:gradient_opf_pf} is replaced with a more accurate nonlinear model $\tilde{\mathbf{r}}^k(\mathbf{p}^k,\mathbf{q}^k) = f(\mathbf{p}^k,\mathbf{q}^k)$.

The computations and updates in steps 4--6 of Algorithm 1 are written more explicitly as
\begin{subequations}\label{eq:gradient_SE_OPF_Nonlinear}
\begin{equation}\label{eq:gradient_voltage_from_SE}
\begin{aligned}
~~~\hat{\mathbf{z}}^k_{\textrm{SE}} = ~& \underset{\mathbf{z}^k}{\text{argmin}}  \frac{1}{2}\left(\mathbf{y}^k -h(\mathbf{z}^k)\right)^\top\mathbf{W}\left(\mathbf{y}^k -h(\mathbf{z}^k)\right), 
\end{aligned}
\end{equation}
\begin{equation}\label{eq:gradient_p}
\hspace{-0.3mm}
\begin{aligned}
\begin{bmatrix}
\mathbf{p}^{k+1}\\
\mathbf{q}^{k+1}
\end{bmatrix} = & 
\begin{bmatrix}\mathbf{p}^{k} - \epsilon \big(\nabla_{\mathbf{p}}C(\mathbf{p}^{k},\mathbf{q}^{k}) + \nabla_{\mathbf{p}}C_0(\mathbf{p}^{k},\mathbf{q}^{k})\\ +\mathbf{A}^\top\nabla^\top_{\mathbf{r}}\mathbf{g}({\color{black}\mathbf{r}^k(\hat{\mathbf{z}}^k_{\textrm{SE}})}), \bm{\mu}^k\big) \\
\mathbf{q}^{k} - \epsilon \big(\nabla_{\mathbf{q}}C(\mathbf{p}^{k},\mathbf{q}^{k}) + \nabla_{\mathbf{q}}C_0(\mathbf{p}^{k},\mathbf{q}^{k})\\
 +\mathbf{B}^\top\nabla^\top _{\mathbf{r}}\mathbf{g}({\color{black}\mathbf{r}^k(\hat{\mathbf{z}}^k_{\textrm{SE}})}), \bm{\mu}^k\big)
\end{bmatrix}_{\mathcal{Z}},
\end{aligned}
\end{equation}
\begin{equation}\label{eq:gradient_mu_upper}
\begin{aligned}
\bm{\mu}^{k+1} = \left[\bm{\mu}^k + \epsilon \left(\mathbf{g}(\mathbf{r}^k(\hat{\mathbf{z}}_{\textrm{SE}}^{k})) - \eta\bm{\mu}^k\right) \right]_{\mathbb{R}_{+}^{N_\mu}},~~~~~~~
\end{aligned}
\end{equation}
\end{subequations}
where $C(\mathbf{p}^k,\mathbf{q}^k):= \sum_{i\in\mathcal{N}}C_i(p_i,q_i)$.
With that we define the gradient mapping with SE in the loop as $\overline{\Phi}(\cdot)$, where the dual variables $\bm{\mu}$ are updated using the SE result from \eqref{eq:gradient_voltage_from_SE}:
\begin{equation}\nonumber
    \overline{\Phi}:\{\mathbf{p}^k,\mathbf{q}^k, \bm{\mu}^k\} \mapsto  \begin{bmatrix}
     \nabla_{\mathbf{p}}\mathcal{L}\left(\mathbf{p}^k,\mathbf{q}^k,\bm{\mu}^k\right) \\
     \nabla_{\mathbf{q}}\mathcal{L}\left(\mathbf{p}^k,\mathbf{q}^k,\bm{\mu}^k\right)\\
    -  \nabla_{\mathbf{\bm{\mu}}}\mathcal{L}\left(\mathbf{r}^k(\mathbf{\hat{p}}^k_{\textrm{SE}},\mathbf{\hat{q}}_{\textrm{SE}}^k),\bm{\mu}^k\right)
    \end{bmatrix}.
\end{equation}
\begin{assumption}\label{assumption_variance}
The squared error in the gradient mapping caused by state estimation has a uniform bound, i.e., there exists $\alpha > 0$ such that
    \begin{equation}\label{eq:SE_Convarnace}
        \mathbf{E}\Big[\|\Phi(\mathbf{x}^k) - \overline{\Phi}(\mathbf{x}^k) \|_2^2\Big]:=\sigma_k^2 \leq \alpha, \quad \forall \mathbf{x}^k.
        \end{equation}
\end{assumption}
\noindent This bound always exists and is realistic due to the physical limits of network states. 
It measures the influence of noisy measurements on the SE in the loop framework, which will be incorporated in our stochastic convergence analysis of Algorithm \ref{algorithm:OPF-SE algorithm}.
At the saddle point of \eqref{eq:maxmin_L}, the squared error caused by state estimation is:
\begin{equation}\nonumber
\sigma_*^2 := \mathbf{E}\Big[\|\Phi(\mathbf{x}^*) - \overline{\Phi}(\mathbf{x}^*) \|_2^2 \Big].
\end{equation}
\begin{assumption}\label{assumption:nonlinear_bounds}
The squared distance between the gradient mapping with SE in-the-loop and that with the nonlinear power flow model in \eqref{eq:nonlinear_feedback} has a uniform bound, i.e., there exists $\rho > 0$ such that 
    \begin{equation}\label{eq:nonlinear_discrepancy}
        \|\overline{\Phi}(\mathbf{x}^k) - \tilde{\Phi}(\mathbf{x}^k) \|_2^2 \leq \rho, \quad \forall \mathbf{x}^k.
    \end{equation}
\end{assumption}
\begin{theorem}\label{theorem_2}
Suppose the step size $\epsilon$ satisfies the condition \eqref{stepsizecond} from Theorem \ref{theorem_1}. Under Assumptions \ref{assumption_variance}--\ref{assumption:nonlinear_bounds}, the sequence $\{\mathbf{x}^k\}$ generated by Algorithm \ref{algorithm:OPF-SE algorithm} satisfies
\begin{equation}\label{eq:upper_bounds_Nonlinear}
    \lim_{k\to\infty} \sup \mathbf{E}\Big[\|\mathbf{x}^{k+1} - \mathbf{x}^*\|_2^2\Big] = \frac{\rho + 3\alpha }{2M/\epsilon  - L^2},
\end{equation}
where $\mathbf{x}^*=[(\mathbf{p}^*)^\top, (\mathbf{q}^*)^\top, (\bm{\mu}^*)^\top ]^\top$ is saddle point of $\mathcal{L}$ in \eqref{eq:maxmin_L}.
\end{theorem}
\begin{proof}
See Appendix.
\end{proof}

The result \eqref{eq:upper_bounds_Nonlinear} from Theorem \ref{theorem_2} provides an upper bound on the expected squared distance between the sequence $\{\mathbf{x}^k | ~ \mathbf{x}^k := [(\mathbf{p}^k)^\top, (\mathbf{q}^k)^\top, (\bm{\mu}^k)^\top ]^\top, k \leq K, K \to \infty \}$ generated by our proposed OPF with SE in-the-loop algorithm \eqref{eq:gradient_SE_OPF_Nonlinear} and the saddle point $\mathbf{x}^*$ of \eqref{eq:maxmin_L}. This analytical bound indicates that our proposed approach has robust performance to estimation errors and measurement noise. 
\begin{itemize}
\item[1.] \emph{Inherent Measurement Noise:} The online measurements by PMUs are typically within $1\% \sim 2\%$ of the actual values. The pseudo-measurements of active and reactive power can be regarded as a rough initialization to time-varying actual values of power (with up to 50\% variations). These errors can be reduced through the estimation phase in \eqref{eq:se_general}, which improves the decisions of the OPF controller with SE feedback \eqref{eq:gradient_voltage_from_SE}, improving robustness to measurement noise and power variability;
\item[2.] \emph{Linearization Approximation Errors:} The update of $[\mathbf{p},\mathbf{q}]$ in \eqref{eq:gradient_p} and maybe also the state estimator \eqref{eq:gradient_voltage_from_SE} utilize the linearized power flow model \eqref{linear_powerflow} to promote computational efficiency. The discrepancy between linearized and nonlinear power flow model is quantified in \eqref{eq:nonlinear_discrepancy} by $\rho$, 
when the electric quantities of interest $\tilde{\mathbf{r}}$ are realized by the nonlinear power flow of the physical network.
\end{itemize}

\noindent To summarize, the discrepancy in \eqref{eq:upper_bounds_Nonlinear} depends on 1) the step size of gradient update, the monotonicity and Lipschitz coefficients of objective and constraint functions; 2) the difference between nonlinear and linearized power flow models, and 3) a general bound to the error in gradient mapping caused by state estimation. Our result reveals that the state estimation errors will not be accumulated with the increasing number of algorithm iterations. 
Moreover, the bound characterized in \eqref{eq:upper_bounds_Nonlinear} can be made arbitrarily small by adopting sufficiently small $\epsilon$.

\begin{remark}(Feedback-based OPF). For an OPF problem in a large-scale distribution network, our method can effectively reduce the computational complexity by adopting the linearized power flow model in the problem formulation, and can improve accuracy by compensating for the modelling error with measurements from the physical network. State estimation bridges the gap between the theoretical design where all the quantities of interest are measured and the reality where only a few nodes are equipped with measuring devices. Analyses for the model-based feedback algorithms to solve OPF can be found in a few recent works \cite{colombino2019online,dall2016optimal}. Extending the convergence analysis to  the nonlinear power flow model will be pursued as a future research effort.
\end{remark}

\subsection{Estimation Error Analysis}\label{se:error_analysis}
In this subsection, we analytically quantify the errors of the SE algorithm under a linearized measurement model:
\begin{equation}\label{eq:se_linear}\nonumber
    \mathbf{y}^k = \mathbf{H}\mathbf{z}^k + \bm{\xi}^k,
\end{equation}
where $\mathbf{H} \in \mathbb{R}^{N_z\times N_y}$ is a measurement matrix that could be obtained by linearizing the nonlinear measurement function around a given nominal operating point. 

The closed-form solution to the linear WLS problem
\begin{equation}\label{eq:}\nonumber
\underset{\mathbf{z}^k}{\min}  \frac{1}{2}\left(\mathbf{y}^k -\mathbf{H}\mathbf{z}^k\right)^\top \mathbf{W}\left(\mathbf{y}^k -\mathbf{H}\mathbf{z}^k\right),
\end{equation}
is $\hat{\mathbf{z}}^k_{\textrm{SE}} = \left(\mathbf{H}^\top \mathbf{W}\mathbf{H}\right)^{-1} \mathbf{H}^\top \mathbf{W}\mathbf{y}^k$. When the matrix $\mathbf{H}^\top \mathbf{W}\mathbf{H}$ is non-singular (which occurs when $\mathbf{W}$ is positive definite and $\mathbf{H}$ has full column rank), the estimate can be expressed as
\begin{equation*}
\begin{aligned}
    \hat{\mathbf{z}}_{\textrm{SE}}^k & = \left(\mathbf{H}^\top \mathbf{W}\mathbf{H}\right)^{-1} \mathbf{H}^\top \mathbf{W} \mathbf{H} \mathbf{z}^k + \left(\mathbf{H}^\top \mathbf{W}\mathbf{H}\right)^{-1}\mathbf{H}^\top \mathbf{W} \bm{\xi}^k\\
& = \mathbf{z}^k + \left(\mathbf{H}^\top \mathbf{W}\mathbf{H}\right)^{-1}\mathbf{H}^\top \mathbf{W} \bm{\xi}^k.
\end{aligned}
\end{equation*}
The WLS estimator is unbiased (since  $\mathbf{E}\left[\hat{\mathbf{z}}_{\textrm{SE}}^k \right] = \mathbf{z}^k$ with zero-mean noise $\bm{\xi}$), and the variance is given by $\textrm{Var}\left[\hat{\mathbf{z}}_{\textrm{SE},j}^k\right] = \sum_{i=1}^n \Gamma_{ji}\sigma_i^2$,
where $\Gamma_{ji}$ denotes the $ji$-th element of $\bm{\Gamma} = \left(\mathbf{H}^\top \mathbf{W}\mathbf{H}\right)^{-1}\mathbf{H}^\top \mathbf{W}$, and $\sigma_i^2$ is the $i$th diagonal element of the measurement covariance matrix $\bf{\Sigma}$.
The confidence interval for state estimate $\hat{\mathbf{z}}^k_{\textrm{SE}}$ can be constructed as
\begin{equation}
\begin{aligned}\nonumber
    & \hat{\mathbf{z}}_{\textrm{SE},j}^k \pm  c \sqrt{\textrm{Var}\left(\hat{\mathbf{z}}_{\textrm{SE},j}^k\right)}  = 
    \hat{\mathbf{z}}_{\textrm{SE},j}^k \pm  c\sqrt{\sum_{i=1}^n \Gamma_{ji}\sigma_i^2},
    \end{aligned}
\end{equation}
where $c$ can be chosen based on the prescribed confidence level. In addition to the bound in Theorem \ref{theorem_2}, the confidence interval provides a numerical performance metric with respect to the state estimation errors. In the next section, we will use this metric to quantify estimation errors caused by noisy voltage measurements and net-loads pseudo-measurements.

\begin{figure}
    \centering
    \includegraphics[width=3in]{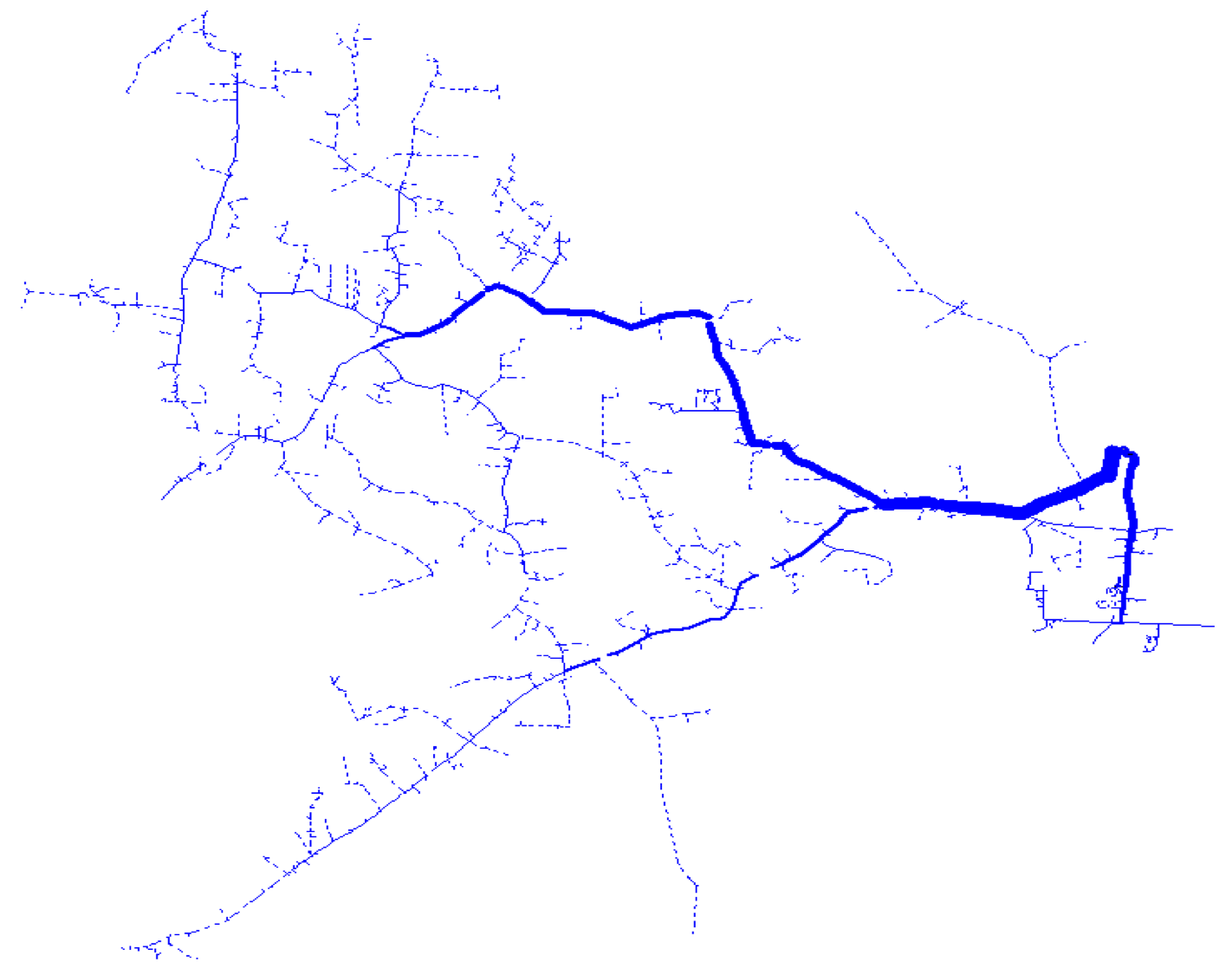}
    \caption{An 11,000-node distribution network. This testbed is constructed by connecting an IEEE 8,500-node distribution network and an EPRI Ckt7 test feeder at PCC. The primary side of this modified feeder is modelled in detail, while the loads on secondary side are merged into distribution transformers. This lumps the 11,000-node testbed into a 4,521-node distribution network.}
    \label{fig:10,000-node testbed}
\end{figure} 

\section{Numerical Results}\label{sec:num}

We conduct numerical experiments to close the loop between OPF and SE. In particular, we explore the tradeoff between the sensing and communication overhead and the performance of OPF controllers in a large network.

We use a modified three-phase unbalanced 4,521-node distribution network (which was reduced from the 11,000-node network in Fig.~\ref{fig:10,000-node testbed} by merging the secondary loads into transformers) to demonstrate the effectiveness and scalability of the proposed OPF solver with SE in the loop. 
This extremely large system is divided into 5 clusters and then we utilized a spatially distributed optimization algorithm for computational affordability. The details of multi-phase power flow modelling and the feasibility of distributed algorithm were discussed in our companion paper \cite{zhou2019accelerated}. Our simulations are conducted on a desktop with AMD Ryzen 7 2700X Eight-Core Processor CPU@3.7GHz, 64GM RAM, Python 3.7 and Windows 10. 

\begin{figure}
    \centering
    \subfigure[normal bound]{\label{fig:voltage_results_095} 
    \includegraphics[width=1.672in]{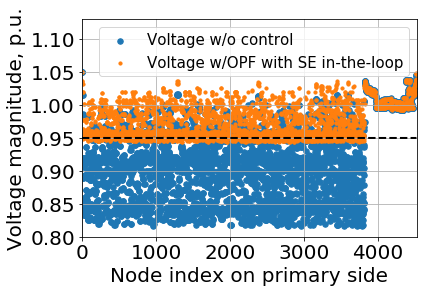}}
    \subfigure[tighter bound]{\label{fig:voltage_results_096} 
    \includegraphics[width=1.672in]{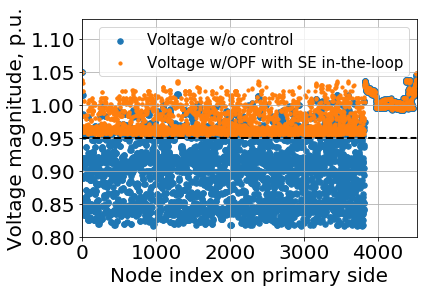}}
    \caption{Voltage profile of OPF controller with SE in the loop. The black dash line indicates the lower voltage bound, i.e., 0.95 p.u.. After we utilize a tighter bound $[0.96,1.05]$ to compensates the inherent errors of SE in the loop, the voltage profile on the right then meets the constraint. }
    \label{fig:control_voltage}
\end{figure}

We consider a voltage regulation problem  where the electrical quantities of interests are voltage magnitudes. In particular, we specify the vector $\mathbf{r}$ as the voltage magnitude $\mathbf{r}:= |\mathbf{v}|:=[|v_1|,\ldots,|v_N|]^\top \in \mathbb{R}^N_{++}$ and consider
\begin{subequations}\label{eq:opt_v}
\begin{eqnarray}\nonumber
\textbf{OPF-V:} & \underset{\mathbf{p},\mathbf{q},|\mathbf{v}|}{\min} & \sum_{i\in\mathcal{N}}C_i(p_i,q_i)+ C_0(\mathbf{p},\mathbf{q}),\nonumber\\
& \text{s.t.}&|\mathbf{v}|=\mathbf{A}\mathbf{p}+\mathbf{B}\mathbf{q}+ |\mathbf{v}_0|,\nonumber\\
&&\underline{\mathbf{v}} \leq  |\mathbf{v}| \leq \overline{\mathbf{v}},\nonumber\\
&& (p_i,q_i)\in\mathcal{Z}_i,\forall i\in\mathcal{N}. \nonumber
\end{eqnarray}
\end{subequations}
The inequality constraints enforce the safety bounds $(\underline{\mathbf{v}},\bar{\mathbf{v}})$ to voltage magnitudes. In particular, we take a linear approximation to AC power flow to express $|\mathbf{v}|$ as a linear function of power injections $(\mathbf{p},\mathbf{q})$. The coefficient matrices $(\mathbf{A},\mathbf{B})$ and constant vector $|\mathbf{v}_0|$ can be attained from numerous linearization methods, e.g., \cite{bernstein2017linear,gan2016online}. The gradient-based OPF controller \eqref{eq:gradient_SE_OPF_Nonlinear} utilizes the online voltage magnitude measurement and voltage estimation to make the system converge. We consider a cost function $C_i(p_i, q_i) = (p_i - p_i^0)^2 + (q_i - q_i^0)^2$ that minimizes the deviation of the power setpoints $(p_i, q_i)$ from their nominal/preferred level $(p_i^0, q_i^0)$ for node $i$. The term $C_0(\mathbf{p}) = \alpha (P_0(\mathbf{p}) - \tilde{P}_0)^2$ penalizes the deviation of the total active power injection $P_0(\mathbf{p})$ at the substation from its preferred values $\tilde{P}_0$. We choose a small weighting factor $\alpha = 0.0005$ to focus on voltage regulation.

The default voltage profile without any control\footnote{We disable all rule-based control of voltage regulators, local capacitors and low-voltage transformers, etc., and only solve the nonlinear power flow.} is calculated by OpenDSS \cite{dugan2010ieee} and shown by blue dots in 
Fig.~\ref{fig:control_voltage}. 
The voltage limits $\overline{\mathbf{v}}$ and $\underline{\mathbf{v}}$ are set to 1.05 and 0.95 p.u. This particular network has a significant under-voltage situation. We implement Algorithm 1 to solve the above voltage regulation problem, while minimizing the objective function. We consider distributed energy resources (DERs) with box constraints. The default net-loads settings can be found in \cite{dugan2010ieee,EPRI2015}. 
The proposed framework can also be adapted to incorporate more complicated models and objectives for energy storage devices, PVs, distributed diesel generators, et cetera. 

Most literature assumed that the full knowledge of real-time voltage is available to a system operator, which requires unrealistic sensor deployment, heavy communication, and huge investment. To balance information availability and measurement overhead, 
we randomly gather voltage magnitude measurements at 3.6\% of the nodes (i.e., about 160 nodes) with zero-mean Gaussian noise of 1\% standard deviation.
We also have pseudo-measurements of all the nodal power injections with zero-mean and huge 50\% standard deviation, which guarantees full observability for state estimation. 
The estimated active and reactive power injections are the primary estimations in the proposed framework. The estimated voltage magnitudes are then calculated from those primary estimations using a nonlinear power flow model. The detailed model for voltage estimation can be found in our prior work \cite{zhou2020gradient}. 
We implement Algorithm 1 with step-size $7 \times 10^{-4}$ for primal updates and $1 \times 10^{-3}$ for dual updates.

Fig.~\ref{fig:control_voltage} visualizes the voltage profile regulated by the proposed SE in-the-loop method with orange dots. In order to prevent voltage from falling below 0.95 p.u., the net-loads must be curtailed based on the SE feedback information. The voltage at most nodes are bounded within $[0.95,1.05]$. There are a few nodes whose voltage magnitudes slightly violate the safety lower bound, which is caused by voltage estimation errors. 
The average and maximum voltage estimation errors are shown in Fig.~\ref{fig:min_max_SE_Errors} for SE in-the-loop.
For comparison, we also test a scenario where the OPF solvers directly utilize raw noisy measurements of all the voltage magnitudes. The noises of those raw measurements are subjected to zero-mean Gaussian distribution with 1\% standard deviation of their true values. 
From Fig.~\ref{fig:min_max_SE_Errors},
we observe that having SE in the loop will significantly reduce the estimation errors, compared to the direct usage of raw measurements. 
Moreover, we utilize the analysis in Section \ref{se:error_analysis} to get Fig.~\ref{fig:SE_errors_CI}, where the average errors under SE in-the-loop are almost always bounded by the analytic 99\% confidence interval over 1000 OPF iterations. 

To resolve the voltage violations caused by estimation errors, we impose a tighter voltage lower bound $0.96$ p.u. based on the analytic confidence interval. As shown in Fig.~\ref{fig:objective_functions}, the network curtails more net-loads to achieve a more conservative voltage profile, which leads to a higher operational cost. We emphasize that there is always a tradeoff between: 1) cost of the measurement system (e.g., number of deployed sensors, communication infrastructure) and 2) OPF controller performance (e.g., robustness, feasibility and optimality). In general, our proposed approach provides utilities and system operators a framework to systematically design OPF controllers under a limited set of sensor measurements.

\begin{figure}[tbhp!]
    \centering
    \includegraphics[width=3.4in]{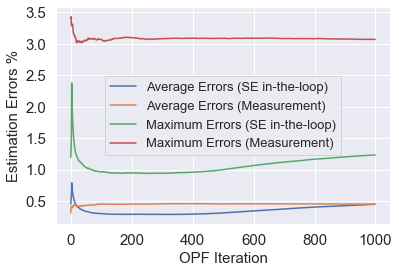}
    \caption{Comparison of estimation errors between SE in the loop and the use of raw voltage measurements. The running average of average/maximum errors show that the SE in the loop yields less error than using raw measurements.}
    \label{fig:min_max_SE_Errors}
\end{figure}

\begin{figure}[tbhp!]
    \centering
    \includegraphics[width=3.4in]{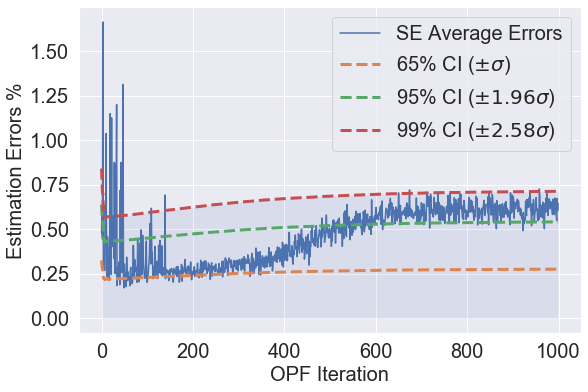}
    \caption{Comparison of the average estimation errors with different confidence intervals over 1000 OPF iterations.}
    \label{fig:SE_errors_CI}
\end{figure}

\begin{figure}[tbhp!]
    \centering
    \includegraphics[width=3.5in]{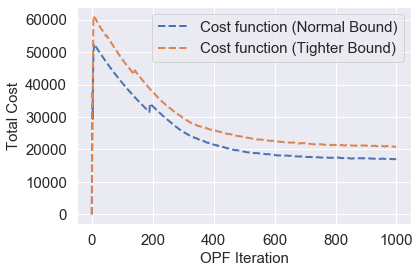}
    \caption{Total cost with SE in the loop over 1000 OPF iterations.}
    \label{fig:objective_functions}
\end{figure}

In summary, numerical results conclude that the proposed OPF controller with SE feedback is able to systemically reduce voltage estimation errors at unmeasured nodes and achieve safe voltage regulation.
The benefits of closing the loop between OPF controllers and SE can be clearly observed from the perspectives of effectiveness, robustness and efficiency. 


\section{Conclusions and Outlooks}\label{sec:con}
In this paper, we proposed a general optimal power flow controller with state estimation feedback to facilitate the operation of modern distribution networks. The controller depends explicitly on the state estimation results derived from system measurements. In contrast to existing works, our method utilizes a feedback loop to the OPF controller to estimate the system voltages from a limited number of sensors rather than making strong assumptions on full observability or requiring full state measurements. The performance of our design is analyzed and numerically demonstrated. The numerical results demonstrate the effectiveness, scalability, and robustness of the proposed OPF controller with SE in the loop. 


Our work launched an initial step to close the loop between control and state estimation in power systems. There are several lines of future work to further explore the benefits and overcome the limitations of this idea, including but not limited to
\begin{itemize}
    \item performance evaluation of various OPF formulations with different SE techniques in the loop;
    \item optimal sensor placement with SE feedback for better OPF performance;
    \item OPF \& SE co-design considering the estimation errors for a more efficient communication structure in a real network;
    \item convergence analysis based on the nonlinear power flow;
    \item performance discussions with different distributed algorithms, e.g., alternating direction method of multipliers (ADMM);
    \item extension to the time-varying power flow linearization.
\end{itemize}

\section*{Acknowledgement}
We would like to thank Dr. Zhiyuan Liu, who was with the Department of Computer Science, the University of Colorado Boulder when the work was done, for his support on the simulation testbed construction. 


\bibliographystyle{ieeetr}  
\bibliography{reference} 

\appendix
\subsection*{Proof of Theorem~\ref{theorem_2}}
\begin{proof}
Now we are ready to show the convergence of Algorithm~\ref{algorithm:OPF-SE algorithm}. The expected distance between the generated sequence $\{\mathbf{x}^{k+1}\}$ and the saddle point of $\mathcal{L}$ can be characterized as:
\begin{equation}\label{eq:inequality_theorem2}
\begin{aligned}
& \mathbf{E}\Big[\|\mathbf{x}^{k+1} - \mathbf{x}^*\|_2^2\Big]\\
\leq ~& \mathbf{E}\Big[\|\mathbf{x}^k - \epsilon \tilde{\Phi}(\mathbf{x}^k) - \mathbf{x}^* + \epsilon \Phi(\mathbf{x}^*)\|_2^2\Big]\\
= ~& \mathbf{E}\Big[\|\mathbf{x}^k - \epsilon \tilde{\Phi}(\mathbf{x}^k) + \epsilon \overline{\Phi}(\mathbf{x}^k) - \epsilon \overline{\Phi}(\mathbf{x}^k)  - \mathbf{x}^* + \epsilon \Phi(\mathbf{x}^*) \\ 
& + \epsilon \overline{\Phi}(\mathbf{x}^*) - \epsilon \overline{\Phi}(\mathbf{x}^*) \|_2^2\Big]\\
\leq ~& \mathbf{E}\Big[\|\mathbf{x}^k -\epsilon\overline{\Phi}(\mathbf{x}^k) - \mathbf{x}^* + \epsilon \overline{\Phi}(\mathbf{x}^*) \|_2^2 + \epsilon^2\|\tilde{\Phi}(\mathbf{x}^k) - \overline{\Phi}(\mathbf{x}^k)\|_2^2 \\
& + \epsilon^2\|\Phi(\mathbf{x}^*) - \overline{\Phi}(\mathbf{x}^*)\|_2^2\Big]\\
= ~& \mathbf{E}\Big[\|\mathbf{x}^k -\epsilon\overline{\Phi}(\mathbf{x}^k) + \epsilon\Phi(\mathbf{x}^k) - \epsilon\Phi(\mathbf{x}^k) - \mathbf{x}^*  + \epsilon \overline{\Phi}(\mathbf{x}^*) \\
& + \epsilon\Phi(\mathbf{x}^*) - \epsilon\Phi(\mathbf{x}^*)\|_2^2 + \epsilon^2\| \tilde{\Phi}(\mathbf{x}^k)-\overline{\Phi}(\mathbf{x}^k) \|_2^2 \\
& + \epsilon^2\|\Phi(\mathbf{x}^*)  - \overline{\Phi}(\mathbf{x}^*)\|_2^2\Big]\\ 
\leq ~& \mathbf{E}\left[\|\mathbf{x}^k -\epsilon\Phi(\mathbf{x}^k) - \mathbf{x}^* + \epsilon \Phi(\mathbf{x}^*) \|_2^2 \right] \\
& + \mathbf{E}\Big[\epsilon^2\|\overline{\Phi}(\mathbf{x}^k) - \Phi(\mathbf{x}^k)\|_2^2 + 2\epsilon^2\|\Phi(\mathbf{x}^*) - \overline{\Phi}(\mathbf{x}^*)\|_2^2\Big]\\
& + \epsilon^2\|\tilde{\Phi}(\mathbf{x}^k) - \bar{\Phi}(\mathbf{x}^k)\|_2^2  \\
\leq ~& \mathbf{E}\left[\|\mathbf{x}^k -\epsilon\Phi(\mathbf{x}^k) - \mathbf{x}^* + \epsilon \Phi(\mathbf{x}^*) \|_2^2\right] + \epsilon^2\left(\rho + \sigma^2_k + 2\sigma_*^2\right)\\ 
\leq ~& \mathbf{E}\left[\|\mathbf{x}^k - \mathbf{x}^*\|_2^2\right] + \mathbf{E}\left[\|\epsilon \Phi(\mathbf{x}^k) - \epsilon\Phi(\mathbf{x}^*) \|_2^2 \right] \\
& - 2\epsilon \left(\Phi(\mathbf{x}^k) - \Phi(\mathbf{x}^*)\right)^\top \left(\mathbf{x}^k - \mathbf{x}^*\right)+ \epsilon^2\left(\rho + \sigma^2_k + 2\sigma_*^2\right)\\
\leq ~& \left(\epsilon^2L^2 - 2\epsilon M + 1 \right)\mathbf{E}\left[\|\mathbf{x}^k - \mathbf{x}^*\|_2^2\right] + \epsilon^2\left(\rho + \sigma^2_k + 2\sigma_*^2\right),
\end{aligned}
\end{equation}
where the first inequality is due to the non-expansiveness of the projection operator, the next two inequalities depend on the triangle inequality, the fourth inequality comes from \eqref{eq:SE_Convarnace}--\eqref{eq:nonlinear_discrepancy}, and the last two inequalities are because of the strong monotonicity \eqref{eq:property_Monotone} and Lipschitz continuity \eqref{eq:property_Lipschitz} of the operator $\Phi$. Next, we let $\Delta = \epsilon^2L^2 - 2\epsilon M +1 $ and recursively implement the steps above until $k=0$ to have
\begin{subequations}
\begin{align}
    &\mathbf{E}\Big[\|\mathbf{x}^{k+1} - \mathbf{x}^*\|_2^2\Big] \nonumber\\ 
    \leq ~ &\Delta^{k+1}\|\mathbf{x}^0 - \mathbf{x}^*\|_2^2 + \epsilon^2\left(\rho + 2\sigma^2_* \right) \left(\frac{1-\Delta^{k+1}}{1-\Delta}\right)\nonumber \\
    \label{eq:proof_SE_tight_bounds} & + \epsilon^2 \sum_{i=0}^{k}\Delta^{k-i}\sigma_k^2 \\
    \leq ~ & \Delta^{k+1}\|\mathbf{x}^0 - \mathbf{x}^*\|_2^2 + \epsilon^2\left(\rho + 2\alpha \right) \left(\frac{1-\Delta^{k+1}}{1-\Delta}\right) \nonumber \\
    & \label{eq:proof_SE_relax_bounds} + \epsilon^2 \frac{1-\Delta^{k+1}}{1-\Delta}\alpha. 
\end{align}
\end{subequations}
By applying the SE estimation variance bound in Assumption \ref{assumption_variance}, \eqref{eq:proof_SE_tight_bounds} is relaxed to \eqref{eq:proof_SE_relax_bounds}. Then we have the step size chosen as $0<\epsilon\leq\bar{\epsilon}<2M/L^2$ from Theorem \ref{theorem_1}, which leads to $0<\Delta\leq \bar{\epsilon}^2L^2 - 2\bar{\epsilon} M +1 <1$. As $k \to \infty$, $\Delta^{k+1}$ on the right-hand-side in \eqref{eq:proof_SE_relax_bounds} will vanish. Given such $\Delta$ and any feasible initial point $\mathbf{x}^0$, we let $k$ approach infinity to get the expectation of this discrepancy as
\begin{equation}\nonumber
    \lim_{k\to\infty} \sup \mathbf{E}\Big[\|\mathbf{x}^{k+1} - \mathbf{x}^*\|_2^2\Big] = \frac{\rho + 3\alpha}{2 M/\epsilon - L^2}.
\end{equation}
This concludes the proof.
\end{proof}

\end{document}